\documentclass[12pt,twoside,reqno]{article}

\usepackage{amsmath}
\usepackage{latexsym}
\usepackage{amsmath,amsthm}
\usepackage{amsfonts}
\usepackage{amssymb}
\usepackage{amsfonts,euscript}
\usepackage{graphicx}
\usepackage{cite}
\usepackage[usenames,dvipsnames]{color}
\def\gianni #1{{\color{red}#1}} 
\def\pier #1{{\color{blue}#1}}
\def\ionut #1{{\color{Green}#1}}
\def\rev #1{{\color{red}#1}}

\def\gianni #1{#1} 
\def\pier #1{#1}
\def\ionut #1{#1}
\def\rev #1{#1}

\DeclareGraphicsExtensions{.eps,epsfig,.bmp,.jpg,.pdf,.mps,.png,.gif}
\newtheorem{theorem}{\bf Theorem}[section]

\newtheorem{proposition}{\bf Proposition}[section]

\newtheorem{lemma}{\bf Lemma}[section]
\newtheorem{remark}{\bf Remark}[section]

\newcommand{\beq}{\begin{equation}}
\newcommand{\eeq}{\end{equation}}
\newcommand{\beqn}{\begin{eqnarray}}
\newcommand{\eeqn}{\end{eqnarray}}
\newcommand{\bear}{\begin{array}}
\newcommand{\eear}{\end{array}}
\newcommand{\beit}{\begin{itemize}}
\newcommand{\eeit}{\end{itemize}}
\newcommand{\beqno}{\begin{eqnarray*}}
\newcommand{\eeqno}{\end{eqnarray*}}

\newcommand{\pdd}[1]{\frac{\partial #1}{\partial\textbf{n}}}

\let\theta\vartheta
\def\eqref #1{(\ref{#1})}

\allowdisplaybreaks

\numberwithin{equation}{section}

\begin{document}
\thispagestyle{plain}
\title{Stabilization of \ionut{a linearized Cahn--Hilliard system for phase separation} by proportional boundary feedbacks}
\author{}
\date{}
\maketitle
\begin{center}
\vskip-1cm
{\large\sc Pierluigi Colli$^{(*)}$}\\
{\normalsize e-mail: {\tt pierluigi.colli@unipv.it}}\\
[.25cm]
{\large\sc Gianni Gilardi$^{(*)}$}\\
{\normalsize e-mail: {\tt gianni.gilardi@unipv.it}}\\
[.25cm]
{\large\sc Ionut Munteanu$^{(**)}$}\\
{\normalsize e-mail: {\tt ionut.munteanu@uaic.ro}}\\
[.45cm]
$^{(*)}$
{\small Dipartimento di Matematica ``F. Casorati'', Universit\`a di Pavia}\\
{\small and Research Associate at the IMATI -- C.N.R. Pavia}\\
{\small via Ferrata 5, 27100 Pavia, Italy}\\
[.25cm]
$^{(**)}$
{\small Alexandru Ioan Cuza University, Department of Mathematics\\
 and Octav Mayer Institute of Mathematics (Romanian Academy)\\ 
  700506 Ia\c{s}i, Romania}
\end{center}

\begin{abstract}
This work \ionut{represents a first contribution on} the problem of boundary stabilization 
\ionut{for} the phase field system of \ionut{Cahn--Hilliard type, which
models} the phase separation in a binary mixture. 
The feedback controller we design here is with actuation only on the temperature flow of the system, on one part of the boundary only. 
Moreover, it is of proportional type, given in an explicit form, expressed only in terms of the eigenfunctions of the Laplace operator, 
being easy to manipulate from the computational point of view. 
Furthermore, it \ionut{ensures} that the closed loop nonlinear system exponentially reaches the prescribed stationary solution provided that the initial \rev{datum} is close enough to~it.

\bigskip

\bigskip 

\end{abstract}

\pagestyle{myheadings}
\newcommand\testopari{\sc Colli \ --- \ Gilardi \ --- \ Munteanu}
\newcommand\testodispari{\sc Stabilization of a Cahn--Hilliard system}
\markboth{\testopari}{\testodispari}

\section{Introduction}

Let $\Omega\subset \mathbb{R}^d$ \gianni{($d=3$ in the applications)} 
be open, bounded, connected, with sufficiently smooth boundary $\Gamma=\partial\Omega$, 
split \gianni{as} $\Gamma=\Gamma_1\cup\Gamma_2,$ \gianni{where} $\Gamma_1$ has nonzero \gianni{surface} measure. 
We consider the boundary controlled \gianni{problem that consists in the} \ionut{Cahn--Hilliard} system
\begin{equation}
\label{t1}
\left\{
  \begin{array}{l}
  (\theta+l_0\varphi)_t-\Delta \theta=0,\ \text{ in }(0,\infty)\times\Omega,
  \\
  \\
  \varphi_t-\Delta \mu=0,\ \text{ in }(0,\infty)\times\Omega,
  \\
  \\
  \mu=-\nu\Delta \varphi+F'(\varphi)-\gamma_0\theta,\ \text{ in }(0,\infty)\times\Omega,
  \end{array}
\right.\ 
\end{equation}
supplemented with the following boundary conditions
\begin{equation}
\label{t2}
\left\{
  \begin{array}{l}\frac{\partial \varphi}{\partial\textbf{n}}=\pdd\mu=0,\ \text{ on }(0,\infty)\times\Gamma,
  \\
  \\
  \pdd\theta=u,\ \text{ on }(0,\infty)\times\Gamma_1,\ \pdd\theta=0,\ \text{ on }(0,\infty)\times\Gamma_2, 
  \end{array} 
\right.\ 
\end{equation}
and with the initial data 
\begin{equation}
\label{t3}
\theta(0)=\theta_o,\ \varphi(0)=\varphi_o,\ \text{in }\Omega.
\end{equation}

In system (\ref{t1})-(\ref{t3}) the variables $\theta,\varphi$ and $\mu$ represent the temperature, 
the order parameter and the chemical potential, respectively; 
$\nu,l_0,\gamma_0$ are positive constants with some physical meaning; 
$F'$ is the derivative of the double-well potential
\begin{equation}
\label{t30}
F(\varphi)=\frac{(\varphi^2-1)^2}{4};
\end{equation}
and $\textbf{n}$ is the unit outward normal vector to the boundary $\Gamma$. 
Finally, $u$ is the control \gianni{acting only on one part of the boundary, namely~$\Gamma_1$}. 
The equations (\ref{t1})-(\ref{t3}) are known as \ionut{the} conserved phase \gianni{field} system, 
due to the mass conservation of~$\varphi$, which is obtained 
by integrating the second equation in (\ref{t1}) 
\gianni{in space and using the boundary condition for $\mu$ from~(\ref{t2})}. 
For more details on conserved phase field system one may check the works \cite{c12,c13,c14,c26}, 
as well the contributions \cite{c17,c30} in which a conserved phase \ionut{field} model, 
allowing further memory effects, is investigated.

Let $(\phi_\infty,\theta_\infty)\in H^4(\Omega)\times H^2(\Omega)$
be any stationary solution of the uncontrolled system (\ref{t1})-(\ref{t3}), i.e.,
\begin{equation}
\left\{
  \begin{array}{l}
  \nu\Delta^2\phi_\infty-\Delta F'(\phi_\infty)=-\Delta \theta_\infty=0,\ \text{ in }\Omega,
  \\
  \\
  \pdd{\phi_\infty}=\pdd{\Delta\phi_\infty}=\pdd{\theta_\infty}=0,\ \text{ on }\Gamma.
  \end{array} 
\right.\ 
\end{equation} 
\ionut{For a discussion on the existence of stationary solutions, see \cite[Lemma~A1]{b1}.}

We emphasize that different stationary profiles correspond to different types of phase \gianni{separation}. 
The aim of the present work is to design a control $u$, in feedback form 
(i.e., expressed as a function of $\varphi$ and~$\theta$) 
such as\ionut{,} once inserted into the equations (\ref{t1})-(\ref{t3}), the corresponding solution satisfies
$$
  \lim_{t\rightarrow\infty}(\varphi(t),\theta(t))=(\varphi_\infty,\theta_\infty)
  \quad \gianni{\hbox{in $L^2(\Omega)\times L^2(\Omega) $}},
$$
with an exponential rate of convergence, provided that 
the initial datum $(\varphi_o,\theta_o)$ is in a suitable neighborhood of $(\varphi_\infty,\theta_\infty)$. 
With other words, we control only the temperature \ionut{flow} 
on the part of the boundary $\Gamma_1$ of $\Gamma$, 
\gianni{in order} that the system behaves similarly as a prescribed type of phase transition. 
\ionut{We aim that it tends to} $(\varphi_\infty,\theta_\infty)$ exponentially fast as time goes on. 

Concerning the problem of stabilization of systems of the type (\ref{t1})-(\ref{t3}), 
we mention the work~\cite{b1}, \gianni{which} provides an internal stabilizing feedback, 
while, concerning the boundary stabilization case, 
the present work represents the first result in this direction. 
\ionut{We follow similar arguments as in \cite{b1}\ionut{,} but we also use ideas from \cite{ionut}, 
where boundary proportional type feedbacks are constructed in order to stabilize equilibrium solutions to parabolic-type equations. Then,}
we aim to design a boundary stabilizing feedback\gianni{, i.e., a map $(\theta,\varphi)\mapsto u$,} 
for $(\varphi_\infty,\theta_\infty)$ in (\ref{t1})-(\ref{t3}).
More precisely, our goal is to prove Theorem~\ref{ter1} below, 
that amounts to saying that the feedback 
\begin{equation}
\label{ie11}
  \begin{aligned}
  &u = u(\theta,\varphi) =
  \left< \ionut{\Lambda_S {\bf A  O}}  
  ,\mathbf{1}\right>_N,
  \end{aligned}
\end{equation}
once plugged into equations (\ref{t1})-(\ref{t3}), yields the exponential stability of them. 
The ingredients of \eqref{ie11} are defined in the next section
(see, in particular, \ionut{equations~(\ref{t20})--(\ref{t25}) and Remark~\ref{REM})}.
Here, we only can roughly summarize as follows.
The term \ionut{${\bf O}$ reads as}
\begin{equation*}
\ionut{{\bf O}}
= \ionut{{\bf O}}(\theta,\varphi)
:=\left(
  \begin{array}{l}           
  \left<\varphi,\phi_1\right>
  +\left<\alpha_0(\theta+l_0\varphi),\psi_1\right>
  -\left<\varphi_\infty,\phi_1\right>
  -\left<\alpha_0(\theta_\infty-l_0\varphi_\infty),\psi_1\right>
  \\
  \\ 
  \left<\varphi,\phi_2\right>
  +\left<\alpha_0(\theta+l_0\varphi),\psi_2\right>
  -\left<\varphi_\infty,\phi_2\right>
  -\left<\alpha_0(\theta_\infty-l_0\varphi_\infty),\psi_2\right>
  \\
  \\ \hskip3cm \dots \ \ \ \dots \ \ \ \dots\\ 
  \\
  \left<\varphi,\phi_N\right>
  +\left<\alpha_0(\theta+l_0\varphi),\psi_N\right>
  -\left<\varphi_\infty,\phi_N\right>
  -\left<\alpha_0(\theta_\infty-l_0\varphi_\infty),\psi_N\right>
  \end{array}
\right)
\end{equation*}
where $\alpha_0:=(\gamma_0/l_0)^{1/2}$,
the symbol $\left<\,\cdot\,,\,\cdot\,\right>$ denotes the standard inner product in~$L^2(\Omega)$
and the system $\left\{(\phi_j,\psi_j)\right\}_{j=1}^N$ 
represents the set of the eigenvectors corresponding to the nonpositive eigenvalues
of the linear operator 
obtained from the linearization of the system (\ref{t1}) around the equilibrium profile $(\varphi_\infty,\theta_\infty)$
after a suitable change of unknown functions. 
Moreover, $A$~is a square constant matrix of order~$N$, 
obtained by some perturbations of the Gram matrix of the system 
$\left\{\psi_j|_{\Gamma_1}\right\}_{j=1}^N$ in $L^2(\Gamma_1)$,
$T$~is a matrix that depends on the traces of the functions $\psi_j,\ j=1,2,\dots,N$, on~$\Gamma_1$,
and $\mathbf{1}$ is the column $N$-vector 
$$\ionut{(1,\dots,1)^T := \left(\begin{array}{c} 1\\ 1\\ \vdots\\ 1 \end{array}\right) .}$$

We emphasize that the stabilizing feedback $u$ is finite-dimensional, linear, given in an explicit form 
that is easy to manipulate from the computational point of view, 
with actuation only on the temperature variable $\theta$, on one part of the boundary only, namely $\Gamma_1$. 
Its design is based on the ideas from the work \cite{ionut}. 
We  mention that, similar proportional-type feedbacks were constructed for stabilizing other important models as: 
the phase field system in \cite{ionut1}, the \rev{3D} periodic channel flow in \cite{ionut6}, 
the \rev{2D} periodic MHD channel flow in \cite{ionut7}, for parabolic-type equations with memory in \cite{ionut2,ionut3}, 
and even for stochastic parabolic equations in \cite{ionut5}.  
We claimed that the proposed feedback is easy to use in numerical simulations. 
This is indeed so. 
The paper \cite{ionut4} provides a numerical example concerning the stabilization of Fischer's equation, proving the efficiency of the controller. \ionut{Let us mention also the contibutions
\cite{pier1, pier2, pier3, pier4} devoted to the analysis of optimal control problems related to Cahn--Hilliard systems with a boundary control.}

\section{Design of a proportional stabilizing feedback}

As in the Introduction, $\Omega\subset \mathbb{R}^d$ is an open set
we assume to be bounded, connected, with sufficiently smooth boundary $\Gamma=\partial\Omega$, 
split as $\Gamma=\Gamma_1\cup\Gamma_2$, where $\Gamma_1$ has a positive surface measure.
We set $H=L^2(\Omega)$ to be the space of all square Lebesgue integrable functions on $\Omega$ 
and we denote by $\|\cdot\|$ and $\left<\,\cdot\,,\,\cdot\,\right>$ its classical norm and scalar product, respectively. 
The same notations are used for the powers of~$H$,
while the symbols $\left<\,\cdot\,,\,\cdot\,\right>_0$ and $\left<\,\cdot\,,\,\cdot\,\right>_N$
stand for the inner products in $L^2(\Gamma_1)$ and $\mathbb{R}^N$, respectively.
We set $H^m(\Omega),\ m>0,$ for the well-known Sobolev spaces of order $m$. 
For a generic normed space~$X$ (different from~$L^2(\Omega)$ or a power of~it), 
$\|\cdot\|_X$ is its corresponding norm.

As in \cite{b1}, we prefer to make a function transformation in (\ref{t1}), namely
\begin{equation}
\label{t5}
\sigma:=\alpha_0(\theta+l_0\varphi),
\end{equation}
with $\alpha_0>0$ chosen such that
\begin{equation}
\label{t6}
\frac{\gamma_0}{\alpha_0}=\alpha_0l_0=:\gamma>0,
\end{equation}
that is
\begin{equation}
\label{t7}
\alpha_0=\sqrt{\frac{\gamma_0}{l_0}}.
\end{equation}
Writing the system (\ref{t1})-(\ref{t3}) in the variables $\varphi$ and $\sigma$ and using (\ref{t6}) and the notation
\begin{equation}
\label{t8}
l:=\gamma_0l_0,
\end{equation}
it yields the equivalent nonlinear system
\begin{equation}
\label{t9}
\left\{
  \begin{array}{l}
  \varphi_t+\nu\Delta^2\varphi-\Delta F'(\varphi)-l\Delta\varphi+\gamma\Delta\sigma=0,\
  \text{ in }(0,\infty)\times\Omega,
  \\
  \\
  \sigma_t-\Delta\sigma+\gamma\Delta\varphi=0,\
  \text{ in }(0,\infty)\times\Omega,
  \\
  \\
  \frac{\partial\varphi}{\partial\mathbf{n}}=0,\ \text{ on }(0,\infty)\times\Gamma,
  \\
  \\
  \frac{\partial\Delta\varphi}{\partial\mathbf{n}}=-\frac{\gamma_0}{\nu}u,\ \frac{\partial\sigma}{\partial\mathbf{n}}=\alpha_0u,\
   \text{ on }(0,\infty)\times\Gamma_1,
  \\
  \\
  \frac{\partial\Delta\varphi}{\partial\mathbf{n}}=\frac{\partial\sigma}{\partial\mathbf{n}}=0,\
  \text{ on }(0,\infty)\times\Gamma_2.
  \end{array}
\right.\ 
\end{equation}

The classical  reduction of the problem, to a null stabilization one,  is done via the fluctuation variables 
\begin{equation}
\label{t10}
y:=\varphi-\phi_\infty,\ z:=\sigma-\sigma_\infty,
\end{equation}
\begin{equation}
\label{t11}
y_o:=\varphi_o-\phi_\infty,\ z_o:=\sigma_o-\sigma_\infty,
\end{equation}
where, clearly,  $\sigma_\infty:=\alpha_0(\theta_\infty+l_0\phi_\infty)$ and $\sigma_o:=\alpha_0(\theta_o+l_0\varphi_o).$  
And so, system  (\ref{t1})-(\ref{t3}) transforms into the equivalent null boundary stabilization problem
\begin{equation}
\label{t12}
\left\{
  \begin{array}{l}
  y_t+\nu\Delta^2 y-\Delta[F'(y+\phi_\infty)-F'(\phi_\infty)]-l\Delta y+\gamma\Delta z=0,
  \ \text{ in }(0,\infty)\times\Omega,
  \\
  \\
  z_t-\Delta z+\gamma \Delta y=0,\ \text{ in }(0,\infty)\times\Omega,
  \\
  \\
  \frac{\partial y}{\partial\mathbf{n}}=0,\ \text{ in }(0,\infty)\times\Gamma,
  \\
  \\
  \frac{\partial\Delta y}{\partial\mathbf{n}}=-\frac{\gamma_0}{\nu}u,\ \frac{\partial z}{\partial\mathbf{n}}=\alpha_0u,\ \text{ on }(0,\infty)\times\Gamma_1,
  \\
  \\
  \frac{\partial\Delta y}{\partial\mathbf{n}}=\frac{\partial z}{\partial\mathbf{n}}=0,\ \text{ on }(0,\infty)\times\Gamma_2,
  \\
  \\
  y(0)=y_o,\ z(0)=z_o.
  \end{array}
\right.\ 
\end{equation}

This is a fourth order differential system due to the presence of $\Delta^2$. Because of the presence of non-linearities under the Laplace operator, the linearized system will not be subtracted from \eqref{t12} in the classical way. More precisely: arguing similarly as in \cite[Eqs. (2.1)-(2.13)]{b1}, we set
\begin{equation}
\label{t13}
F_l:=\overline{F''_\infty}+l,
\end{equation}
where
\begin{equation}
\label{t14}
\overline{F_\infty''}:=\frac{1}{m_\Omega}\int_\Omega F''(\phi_\infty(\xi))d\xi,
\end{equation}
with $m_\Omega$ the Lebesgue measure of $\Omega$, and introduce the linear system
\begin{equation}
\label{t15}
\left\{
  \begin{array}{l}
  y_t+\nu\Delta^2 y-F_l\Delta y+\gamma\Delta z=0 \ \text{ in }(0,\infty)\times\Omega,
  \\
  \\
  z_t-\Delta z+\gamma \Delta y=0\ \text{ in }(0,\infty)\times\Omega,
  \\
  \\
  \frac{\partial y}{\partial\mathbf{n}}=0\ \text{ in }(0,\infty)\times\Gamma,
  \\
  \\
  \frac{\partial\Delta y}{\partial\mathbf{n}}=-\frac{\gamma_0}{\nu}u,\ \frac{\partial z}{\partial\mathbf{n}}=\alpha_0u\ \text{ on }(0,\infty)\times\Gamma_1,
  \\
  \\
  \frac{\partial\Delta y}{\partial\mathbf{n}}=\frac{\partial z}{\partial\mathbf{n}}=0\ \text{ on }(0,\infty)\times\Gamma_2,
  \\
  \\
  y(0)=y_o,\ z(0)=z_o.
  \end{array}
\right.\ 
\end{equation}
We remark that the above system is not the linearization of (\ref{t12}),
since the replacement of the nonlinear term is different from the usual one.

\subsection{Stabilization of the linearized system}

Set $\mathcal{A}:\mathcal{D}(\mathcal{A})\subset L^2(\Omega)\times L^2(\Omega)\rightarrow L^2(\Omega)\times L^2(\Omega),$ 
\begin{equation}
\label{bebe2}
\mathcal{A}:=\left[
  \begin{array}{cc}
  \nu \Delta^2-F_l\Delta& \gamma\Delta
  \\
  \gamma\Delta&-\Delta
  \end{array}
\right],
\end{equation}
having the domain
\rev{\begin{equation}
\label{gianni1}
\begin{aligned}
  D(\mathcal{A})=&\left\{
    (y\ z)^T\in L^2(\Omega)\times L^2(\Omega):\ 
    \mathcal{A}(y\ z)^T\in L^2(\Omega)\times L^2(\Omega),\right.
    \\
    &\hskip5cm \left.\hbox{$\frac{\partial y}{\partial\textbf{n}}
    =\frac{\partial \Delta y}{\partial\textbf{n}}
    =\frac{\partial z}{\partial\textbf{n}}$}
    =0 \text{ on }\Gamma
  \right\}
\end{aligned}
\end{equation}}%
endowed with its graph norm.
By the regularity of $\Omega$ it follows that 
$$\mathcal{D}(\mathcal{A})~\subset~ H^4(\Omega)~\times~ H^2(\Omega).$$ 
Also, we notice that $\mathcal{A}$ is self-adjoint. 
By \cite[Proposition 2.1]{b1}, we know that $\mathcal{A}$ is quasi $m$-accretive on $L^2(\Omega)\times L^2(\Omega)$ and its resolvent is compact. 
Therefore, $\mathcal{A}$ has a countable set 
$\left\{\lambda_j\right\}_{j=1}^\infty$ of real eigenvalues and a complete set of corresponding eigenvectors.
Moreover, all the eigenspaces are finite-dimensional and, by repeating each eigenvalue accordingly to its order of multiplicity,
we have that
\begin{equation}
  \lambda_1\leq\lambda_2\leq\lambda_3\leq\dots
  \text{ and }
  \lim_{j\to\infty} \lambda_j = +\infty.
  \label{eigenvalues}
\end{equation}We notice that  zero is an eigenvalue, and it is of multiplicity \ionut{greater than or equal to~$2$,} since 
$$\sqrt{\frac{1}{2m_\Omega}}\,(1\ 1)^T\ \text{ and } \ \sqrt{\frac{1}{2m_\Omega}}\,(-1\ 1)^T$$ are \ionut{two of} its corresponding eigenvectors. In what follows, \rev{we assume that 
\begin{equation}
\label{pier1} 
\mathbf{(H_0):}\qquad 
\text{the multiplicity of the null eigenvalue is equal to $2$.}
\end{equation}
This assumption is commented in the Remark~\ref{GIANNI} below.}%

By~\eqref{eigenvalues}, the number of nonpositive eigenvalues is finite, i.e.,
for some $N\in\mathbb{N}$, we have that
\begin{equation}
  \lambda_j< 0,\ j=1,2,...,N-2,\ \lambda_{N-1}=\lambda_N=0 \text{ and }  \lambda_j>0 \text{ for }j>N.  
\end{equation}
\rev{Although it would be possible to treat the case of semi-simple negative eigenvalues
following \cite{ionut}, for} the sake of simplicity, we assume that 
\begin{equation}
\label{t18} 
\mathbf{(H_1):}\qquad
\text{\ionut{each} negative eigenvalue is simple.}
\end{equation}
Denote by $\left\{(\varphi_j\ \psi_j)^T\right\}_{j=1}^\infty$ the corresponding eigenvectors, that is
\begin{equation}
\label{t17}
\left\{
  \begin{array}{l}
  \nu \Delta ^2\varphi_j-F_l\Delta\varphi_j+\gamma\Delta\psi_j=\lambda_j\varphi_j,\ \text{ in }\Omega,
  \\
  \\
  \gamma\Delta\varphi_j-\Delta\psi_j=\lambda_j\psi_j,\ \text{ in }\Omega,
  \\
  \\
  \frac{\partial\varphi_j}{\partial\mathbf{n}}=\frac{\partial\Delta\varphi_j}{\partial\mathbf{n}}=\frac{\partial\psi_j}{\partial\mathbf{n}}=0,\ \text{ on }\Gamma,
  \end{array}
\right.\ 
\end{equation} 
for all $j=1,2,\dots$. 

By the self-adjointness of $\mathcal{A}$ we may assume that the system 
$\left\{(\varphi_j\ \psi_j)^T\right\}_{j=1}^\infty$ 
forms an orthonormal basis in~$L^2(\Omega)\times L^2(\Omega)$ and is orthogonal in~$\mathcal{D}(\mathcal{A})$.

\begin{remark}\label{GIANNI}
\rev{The assumption $\mathbf{(H_0)}$ is related to the eigenvalue problem for~$-\Delta$ (with Neumann boundary conditions, of course), as we explain at once. The eigenvalue problem for the null eigenvalue of $\mathcal{A}$ reads
\begin{equation}
\label{gianni2}
\nu \Delta ^2\varphi-F_l\Delta\varphi+\gamma\Delta\psi=0 
  \quad \hbox{and} \quad
  \gamma\Delta\varphi-\Delta\psi=0 
\end{equation} 
with the boundary conditions written in \eqref{gianni1}. Since the second equation is equivalent to $\Delta (\gamma \varphi - \psi)=0$, i.e., $\psi = \gamma\varphi + C$, where $C$ is an 
arbitrary constant, the above problem reduces~to 
$$
 \Delta\bigl(\nu \Delta \varphi +(\gamma^2 - F_l)\varphi \bigr)=0.
$$
As $\frac{\partial}{\partial\mathbf{n}}\bigl(\nu \Delta \varphi +(\gamma^2 - F_l)\varphi \bigr)=0$, this means that
\begin{equation}
\label{gianni3}
-  \Delta\varphi = \bar{\lambda} \varphi + C_0 
\quad \hbox{with}\quad \bar{\lambda}:= \frac{\gamma^2 - F_l}\nu \, ,
\end{equation} 
where $C_0$ is an arbitrary constant. So, there are three cases.}

\smallskip\noindent
\rev{$i)$
$\bar{\lambda}$ is not an eigenvalue of $-\Delta$. Then \eqref{gianni3} has for any $C_0$ a
unique solution $\varphi$ and the solution to \eqref{gianni2} are given by $(\varphi \ \ \gamma \varphi + C)^T $. Thus, the eigenfunctions depend on the two independent constants $C_0$ and $C$, and $\mathbf{(H_0)}$ is satisfied.}

\smallskip\noindent
\rev{$ii)$ $\bar{\lambda}$ is a simple eigenvalue of $-\Delta$. Then, the constant $C$ in 
\eqref{gianni3} must be zero and the solutions to \eqref{gianni3} form a one-dimensional 
space. In this case, we have a similar situation as in~$i)$ and $\mathbf{(H_0)}$ is satisfied.}

\smallskip\noindent
\rev{$iii)$ $\bar{\lambda}$ is an eigenvalue of $-\Delta$ with multiplicity $m>1$. Again, $C$ in \eqref{gianni3} must be zero, but the solutions to \eqref{gianni3} depend on $m$ independent functions $\varphi^1, \dots, \varphi^m$. Then, the eigenspace of \eqref{gianni2}
is generated by the independent pairs 
$$
( \varphi^i \  \ \gamma \varphi^i + C_j)^T, \quad i,j=1, \dots , m. 
$$
So, its dimension is 
$2m >2$ and $\mathbf{(H_0)}$ fails.}

\smallskip\noindent
\rev{Of course, the case $iii)$ might occur and it is even easy to construct examples. However, it is an exceptional situation which can be avoided by a small perturbation of the coefficients. Therefore, $\mathbf{(H_0)}$~is generally satisfied.}
\end{remark}

The control design procedure, developed in \cite{ionut}, requires further knowledge about the 
eigenvectors of the liner operator $\mathcal{A}$. We refer to  the unique continuation 
property, of the eigenvectors, that will be described below. It is clear  by the form of the 
operator $\mathcal{A}$, which involves the Laplace operator, that the eigenvectors  $
(\varphi_j\ \psi_j)^T$ can be put in connection with the eigenfunctions of the Neumann-
Laplacian (this is indeed true, see \eqref{t40} below). In this light, let us denote by $
\left\{\mu_j\right\}_{j=1}^\infty$ and by $\left\{e_j\right\}_{j=1}^\infty$ 
the eigenvalues and the normalized eigenfunctions of the Neumann-Laplacian, respectively, 
i.e.,
$$
  \Delta e_j=\mu_j e_j \text{ in }\Omega
  \text{ and }
  \frac{\partial e_j}{\partial\mathbf{n}}=0 \text{ on }\Gamma,
$$
to which we simply refer to as the Laplace operator $\Delta$ in the sequel.
(We took $\Delta$ instead of $-\Delta$ to be in concordance with the form of the operator $\mathcal{A}$, 
that was firstly introduced in \cite{b1}.) 
We know that $\mu_j\leq0$ for all $j=1,2,\dots$, $\mu_j\rightarrow -\infty$ for $j\rightarrow \infty$. 
Moreover, $\left\{e_j\right\}_{j=1}^\infty$ forms an orthonormal basis 
in~$L^2(\Omega)$ which is orthogonal in~$H^1(\Omega)$.

We have the following result \ionut{which amounts to say that we can choose the eigenbasis
$\left\{(\varphi_j\ \psi_j)^T\right\}_{j=1}^N$ such that the property~\eqref{t44} below holds.}

\begin{lemma}\label{erol1}For all $j=1,2,...,N-2$,  there exists some eigenfunction $e_k$ of the Laplace operator, corresponding to the eigenvalue $\mu_k$, \ionut{which} satisfies
\begin{equation}
\label{nec-muk-1}
0>\mu_k\geq\frac{F_l-\gamma^2}{\nu},
\end{equation} 
 \ionut{and we can choose $\psi_j$ of the following form} 
\begin{equation}\label{t40}\psi_j\equiv\frac{\gamma\mu_k}{\sqrt{(\gamma\mu_k)^2+(\lambda_j+\mu_k)^2}}e_k.\end{equation} In this case the eigenvalue $\lambda_j$ is a root of the second  degree polynomial
$$X^2+[(F_l+1)\mu_k-\nu\mu_k^2]X-\nu\mu_k^3+(F_l-\gamma^2)\mu_k^2.$$

In particular, we deduce that necessarily $F_l-\gamma^2\leq 0$, in order \ionut{to have negative eigenvalues} for the operator $\mathcal{A}$; and
\begin{equation}\label{t44}\psi_j \not \equiv 0 \text{ on }\Gamma_1,\ \forall j=1,2,...,N.\end{equation}
\end{lemma} 
\begin{proof} Let any $j\in\left\{1,2,...,N-2\right\}.$ For the sake of simplicity of the notations we drop the indices $j$, that is we use the notations
$$(\varphi\ \psi)^T=(\varphi_j\ \psi_j)^T \text{ and }\lambda=\lambda_j.$$ We have that $(\varphi\ \psi)^T$ satisfies
\begin{equation}\label{t19}\left\{\begin{array}{l}\nu \Delta ^2\varphi-F_l\Delta\varphi+\gamma\Delta\psi=\lambda\varphi,\ \text{ in }\Omega,\\
\\
\gamma\Delta\varphi-\Delta\psi=\lambda\psi,\ \text{ in }\Omega,\\
\\
\frac{\partial\varphi}{\partial\mathbf{n}}=\frac{\partial\Delta\varphi}{\partial\mathbf{n}}=\frac{\partial\psi}{\partial\mathbf{n}}=0,\ \text{ on }\Gamma.\end{array}\right.\ \end{equation}

Let us decompose $\varphi$ and $\psi$ in the basis $\left\{e_j\right\}_{j=1}^\infty$ of the eigenfunctions of the Neumann Laplacean, as
$$\varphi=\sum_{j=1}^\infty \varphi^je_j,\ \psi=\sum_{j=1}^\infty \psi^je_j.$$ Successively scalarly multiplying \ionut{equations} (\ref{t19}) by $e_j,\ j=1,2,...$, using the boundary conditions and the \ionut{Green} formula, we deduce that
$$\left\{\begin{array}{l}(\nu\mu_j^2-F_l\mu_j-\lambda)\varphi^j+\gamma\mu_j\psi^j \ \  \ =0,\\
 \ \ \  \ \ \ \gamma\mu_j\varphi^j \ \ \ - \ \ \ \ \ (\lambda+\mu_j)\psi^j=0,\ \forall j\in \mathbb{N}^*.\end{array}\right.\ $$ For all $j$, this is a second order linear homogeneous system, with the unknowns $\varphi^j,\psi^j$. Computing the determinant of the matrix of the system, we get that 
\begin{equation}\label{t45}-\nu (\mu_j)^3+(F_l-\gamma^2-\nu\lambda)(\mu_j)^2+\lambda(F_l+1)\mu_j+\lambda^2=0,
\end{equation}
\ionut{for every $j$ such that $(\phi^j\ \psi^j)^T\not=(0\ 0)^T $.
On the other hand, since $(\phi\ \psi)^T\not=(0\ 0)^T $,
\eqref{t45} must hold for at least one value of~$j$,
and we now deal with such values.
Let us consider the polynomial}
\begin{equation}
  \gianni{P(X):={}} -\nu X^3+(F_l-\gamma^2-\nu\lambda)X^2+\lambda(F_l+1)X+\lambda^2 \gianni.
  \label{defP}
\end{equation}
\ionut{Since the product of its (complex) roots is $\lambda^2/\nu>0$ 
and one of them is the nonpositive real number~$\mu_j$,
we infer that $\mu_j<0$, which is one of the inequalities in~\eqref{nec-muk-1}.
Moreover, $P(0)=\lambda^2>0$ and $\lim_{X\to+\infty}P(X)=-\infty$,
so that $P(X)$ has a positive root.
It follows that all the roots are real, non-zero and that just one of them is positive.
At this point, there are several cases, namely:
there are two different negative roots $X_1$ and $X_2$ and both of them are eigenvalues of the Laplace operator;
there is just one negative root (necessarily a double root of the polynomial and an eigenvalue of~$\Delta$);
just one of the negative roots is an eigenvalue of~$\Delta$.
We treat only the first situation,
which is the ``worst scenario''; the other cases can be treated similarly. 
The relation \eqref{t44} holds as a consequence, 
since the eigenfunctions of the Neumann Laplacean cannot vanish identically on~$\Gamma_1$.
Denote by $X_1,X_2$ the two neagtive roots of the polynomial from \eqref{defP}.}
Assume that we have
$$\mu_k=\mu_{k+1}=...=\mu_{k+M}=X_1,$$and
$$\mu_{s}=\mu_{s+1}=...=\mu_{s+L}=X_2,$$i.e., $X_1$ is an eigenvalue of the Laplace operator of multiplicity $M+1$, and $X_2$ is an eigenvalue of the Laplace operator of multiplicity $L+1$. Of course, it may happen that only one of $X_1, X_2$ 
is an eigenvalue of the Laplace operator (we see that at least one is an eigenvalue in order not to have $\varphi\equiv\psi\equiv0$, that contradicts with the fact that $(\varphi\ \psi)^T$ is an eigenvector). In that case, all the below discussion can be easily reconsidered, and one should arrive to similar conclusions. However, we shall consider the "worst case scenario": $X_1,X_2$ are eigenvalues of the \ionut{Laplacean}. 

One can easily check that  each vector from the systems
$$\mathcal{X}_1:=\left\{\left(\frac{\lambda+X_1}{\sqrt{(\gamma X_1)^2+(\lambda+X_1)^2}}e_q\ \ \frac{\gamma X_1}{\sqrt{(\gamma X_1)^2+(\lambda+X_1)^2}}e_q\right)^T \!\!\!\!,\ q=k,k+1,...,k+M\right\}\!,$$ and
$$\mathcal{X}_2:=\left\{\left(\frac{\lambda+X_2}{\sqrt{(\gamma X_2)^2+(\lambda+X_2)^2}}e_q\ \ \frac{\gamma X_2}{\sqrt{(\gamma X_2)^2+(\lambda+X_2)^2}}e_q\right)^T\!\!\!\!,\ q=s,s+1,...,s+L\right\}\!,$$ has unit norm and satisfies equation (\ref{t19}), i.e., it is an eigenvector for $\mathcal{A}$, corresponding to the eigenvalue $\lambda$. Notice that $\mathcal{X}_1\cup \mathcal{X}_2$ contains $M+L+2$    orthogonal unit vectors,  that are, in particular, linearly independent.

 Furthermore, arguing as above, let any $(\tilde{\varphi}\ \tilde{\psi})^T$ verifying  (\ref{t19}), then necessarily 
$$\tilde{\varphi}=\tilde{\varphi}^{k}e_k+\tilde{\varphi}^{k+1}e_{k+1}+...+\tilde{\varphi}^{k+M}e_{k+M}+\tilde{\varphi}^{s}e_s+\tilde{\varphi}^{s+1}e_{s+1}+...+\tilde{\varphi}^{s+L}e_{s+L},$$
and
$$\tilde{\psi}=\tilde{\psi}^{k}e_k+\tilde{\psi}^{k+1}e_{k+1}+...+\tilde{\psi}^{k+M}e_{k+M}+\tilde{\psi}^{s}e_s+\tilde{\psi}^{s+1}e_{s+1}+...+\tilde{\psi}^{s+L}e_{s+L},$$that is $\tilde{\varphi}$ and $\tilde{\psi}$ are linear combinations of the eigenfunctions 
$$\left\{e_k,e_{k+1},...,e_{k+M},e_s,e_{s+1},...,e_{s+L}\right\}.$$ 

Taking into account that the system $\left\{e_k,e_{k+1},...,e_{k+M},e_s,e_{s+1},...,e_{s+L}\right\}$ is linearly independent, plugging the above $\tilde{\varphi}$ and $\tilde{\psi}$ into relations (\ref{t19}), and recalling that $\mu_k=...=\mu_{k+M}=X_1,\ \mu_{s}=\mu_{s+1}=...=\mu_{s+L}=X_2$, we deduce that
$$\tilde{\varphi}^{q}=\frac{\lambda+X_1}{\gamma X_1}\tilde{\psi}^q,\ q=k,k+1,...,k+M,$$and
$$\tilde{\varphi}^{q}=\frac{\lambda+X_2}{\gamma X_2}\tilde{\psi}^q,\ q=s,s+1,...,s+L.$$ 
Hence
$$\begin{aligned} (\tilde{\varphi}\ \tilde{\psi})^T=&\tilde{\psi}^k\left(\frac{\lambda+X_1}{\gamma X_1}e_k\ e_k\right)^T+...+\tilde{\psi}^{k+M}\left(\frac{\lambda+X_1}{\gamma X_1}e_{k+M}\ e_{k+M}\right)^T\\&+\tilde{\psi}^s\left(\frac{\lambda+X_2}{\gamma X_2}e_s\ e_s\right)^T+...+\tilde{\psi}^{s+L}\left(\frac{\lambda+X_2}{\gamma X_2}e_{s+L}\ e_{s+L}\right)^T.\end{aligned} $$
Or, equivalently,
$$\begin{aligned}&(\tilde{\varphi}\ \tilde{\psi})^T\\
&=\frac{\sqrt{(\lambda+X_1)^2+(\gamma X_1)^2}}{\gamma X_1}\tilde{\psi}^k\left(\frac{\lambda+X_1}{\sqrt{(\lambda+X_1)^2+(\gamma X_1)^2}}e_k\ \frac{\gamma X_1}{\sqrt{(\lambda+X_1)^2+(\gamma X_1)^2}}e_k\right)^T+...\\&
+\frac{\sqrt{(\lambda+X_1)^2+(\gamma X_1)^2}}{\gamma X_1}\tilde{\psi}^{k+M}\left(\frac{\lambda+X_1}{\sqrt{(\lambda+X_1)^2+(\gamma X_1)^2}}e_{k+M}\ \frac{\gamma X_1}{\sqrt{(\lambda+X_1)^2+(\gamma X_1)^2}}e_{k+M}\right)^T\\&
+\frac{\sqrt{(\lambda+X_2)^2+(\gamma X_2)^2}}{\gamma X_2}\tilde{\psi}^s\left(\frac{\lambda+X_2}{\sqrt{(\lambda+X_2)^2+(\gamma X_2)^2}}e_s\ \frac{\gamma X_2}{\sqrt{(\lambda+X_2)^2+(\gamma X_2)^2}}e_s\right)^T+...\\&
+\frac{\sqrt{(\lambda+X_2)^2+(\gamma X_2)^2}}{\gamma X_2}\tilde{\psi}^{s+L}\left(\frac{\lambda+X_2}{\sqrt{(\lambda+X_2)^2+(\gamma X_2)^2}}e_{s+L}\ \frac{\gamma X_2}{\sqrt{(\lambda+X_2)^2+(\gamma X_2)^2}}e_{s+L}\right)^T.\end{aligned}$$
Thus, we obtain that the above $(\tilde{\varphi}\ \tilde{\psi})^T$ may be written as a linear combination of the vectors from $\mathcal{X}_1\cup\mathcal{X}_2$. In other words $\mathcal{X}_1\cup\mathcal{X}_2$ forms a system of generators  for the subspace of the eigenvectors of the operator $\mathcal{A}$ corresponding to the eigenvalue $\lambda$. Recalling that  $\mathcal{X}_1\cup\mathcal{X}_2$ is linearly independent we conclude that, in fact, $\mathcal{X}_1\cup\mathcal{X}_2$ represents an orthonormal  basis of this subspace. Consequently, we may choose the eigenvector $(\varphi \ \psi)^T$ to be one of the elements from $\mathcal{X}_1\cup\mathcal{X}_2$, in particular, such that $\psi$ is of the form (\ref{t40}). 

\ionut{Now, we have the inequality $$\mu_k\geq \frac{F_l - \gamma^2}{\nu}$$ claimed in the statement. To do that, let} us consider now the identity (\ref{t45}) as one of unknown $\lambda$. So, $\lambda$ must be a root of the second degree polynomial
$$X^2+[(F_l+1)\mu_k-\nu\mu_k^2]X-\nu\mu_k^3+(F_l-\gamma^2)\mu_k^2.$$ We observe that the above polynomial has  a negative root, namely $\lambda$, and a nonnegative one. Indeed, assume by contradiction that both roots are negative. Then, by the \ionut{Vi\`ete} relations, we deduce that
$$-\nu\mu_k^3+(F_l-\gamma^2)\mu_k^2>0$$and
$$(F_l+1)\mu_k-\nu\mu_k^2>0.$$ Adding to the first relation, the second one, multiplied by $-\mu_k\geq 0$ , we deduce
$$-\mu_k^2(1+\gamma^2)>0,$$ which is absurd. Therefore, again by the \ionut{Vi\`ete} relations, we get that necessarily
$$-\nu\mu_k^3+(F_l-\gamma^2)\mu_k^2\leq0,$$ therefore $\mu_k$ must satisfy $\mu_k\geq\frac{F_l-\gamma^2}{\nu}$\ionut{, as claimed.}
\end{proof}

Following the ideas in \cite{ionut}, we transform the boundary control problem into an 
internal-type one, by lifting the boundary conditions into equations. 
In order to do this, let us define the so-called Neumann operator, as: 
given $a\in L^2(\Gamma_1)$ and $\eta>0$, we set $\ionut{D}_\eta a:=(y\ z)^T,$ the solution to the system 
\begin{equation}
\label{ero6}
\left\{
  \begin{array}{l}
  \begin{aligned} \nu\Delta^2 y&-F_l\Delta y+\gamma\Delta z-2\sum_{j=1}^N\lambda_j\left<(y,z),(\varphi_j,\psi_j)\right>\varphi_j\\
  &-\delta\left<(y,z),(\varphi_N,\psi_N)\right>\varphi_N+\eta y=0,\ \text{ in }\Omega,
  \end{aligned}
  \\
  \\
\begin{aligned}-\Delta z&+\gamma\Delta y-2\sum_{j=1}^N\lambda_j\left<(y,z),(\varphi_j,\psi_j)\right>\psi_j\\
&-\delta\left<(y,z),(\varphi_N,\psi_N)\right>\psi_N+\eta z=0,\ \text{ in }\Omega,
\end{aligned}
  \\
  \\
  \frac{\partial y}{\partial\mathbf{n}}=0,\ \text{ in }(0,\infty)\times\Gamma,
  \\
  \\
  \frac{\partial\Delta y}{\partial\mathbf{n}}=-\frac{\gamma_0}{\nu}a,\ \frac{\partial z}{\partial\mathbf{n}}=\alpha_0a,\ \text{ on }(0,\infty)\times\Gamma_1,
  \\
  \\
  \frac{\partial\Delta y}{\partial\mathbf{n}}=\frac{\partial z}{\partial\mathbf{n}}=0,\ \text{ on }(0,\infty)\times\Gamma_2,
  \end{array}
\right.\ 
\end{equation}  
where $\delta>0$ is such that $\lambda_1,...,\lambda_{N-1},\lambda_N+\delta$ are mutually distinct \ionut{(recall that $\lambda_{N-1} = \lambda_{N} =0 $)}; and $\eta$ is sufficiently large in order to ensure the existence of a unique solution to (\ref{ero6}). So, we define the map 
$$\ionut{D}_\eta\in \rev{\mathcal{L}}\bigl(L^2(\Gamma_1), H^1(\Omega)\times H^{1/2}(\Omega)\bigr).$$

Easy computations, involving the Green's formula, lead to
\begin{equation}
\label{ero7}\begin{split}&
\left<\ionut{D}_\eta a,(\varphi_j\ \psi_j)^T\right>=\frac{\alpha_0}{\eta-\lambda_j}\left<a,\psi_j\right>_0,\ j=1,2,\dots,N-1,\\&
\left<\ionut{D}_\eta a,(\varphi_N\ \psi_N)^T\right>=\frac{\alpha_0}{\eta-\lambda_N-\delta}\left<a,\psi_N\right>_0
\end{split}\end{equation}
where, we recall that,  
$\left<{}\cdot{},{}\cdot{}\right>_0$ stands for the classical scalar product in~$L^2(\Gamma_1)$.

Let 
$$
  0<\eta_1<\eta_2:=\eta_1+\frac{1}{N-1}<\eta_3:=\eta_1+\frac{1}{N-2}<\dots<\eta_N:=\eta_1+1
$$
be $N$ constants sufficiently large such as (\ref{ero6}) is well-posed for each of them.
For a future use, we set
\begin{equation}\label{defDeta}
  \hbox{$\ionut{D}_{\eta_i},\ i=1,2,\dots,N$,  the corresponding solutions of (\ref{ero6}). } 
\end{equation} 

Further, set 
\begin{equation}
\label{t20}
\Lambda_{\eta_k}:=\text{diag}
\left(
  \frac{1}{\eta_k-\lambda_1},\frac{1}{\eta_k-\lambda_2},\dots,\frac{1}{\eta_k-\lambda_{N-1}},\frac{1}{\eta_k-\lambda_N-\delta}
\right),\ k=1,2,\dots,N,
\end{equation} and
$$\Lambda_S:=\sum_{k=1}^N\Lambda_{\eta_k}.$$
Moreover, denote by
\begin{equation}
\label{t21}
B_k:=\Lambda_{\eta_k}\mathbf{B}\Lambda_{\eta_k},\ k=1,2,\dots,N,
\end{equation} 
where $\mathbf{B}$ is the Gram matrix of the system 
$\left\{\psi_j|_{_{\Gamma_1}}\right\}_{j=1}^N,$ in $L^2(\Gamma_1)$, i.e.,
\begin{equation}
\label{t22}
\mathbf{B}:=\left(
  \begin{array}{cccc}
  \left<\psi_1,\psi_1\right>_0
  & \left<\psi_1,\psi_2\right>_0&\dots&\left<\psi_1,\psi_N\right>_0
  \\
  \left<\psi_2,\psi_1\right>_0& \left<\psi_2,\psi_2\right>_0&\dots&\left<\psi_2,\psi_N\right>_0
  \\
  \dots&\dots&\dots&\dots
  \\
  \left<\psi_N,\psi_1\right>_0& \left<\psi_N,\psi_2\right>_0
  &\dots&
  \left<\psi_N,\psi_N\right>_0
  \end{array}
\right).
\end{equation}
Set 
\begin{equation}
\label{t25}
(B_1+B_2+\dots+B_N)^{-1}=:\mathbf{A}.
\end{equation}
One can show that the sum $B_1+B_2+...+B_N$ is indeed invertible by arguing similarly as in the Appendix in \cite{ionut}, and making use of Lemma \ref{erol1}, which states that, for all $j=1,2,\dots,N$, the trace of $\psi_j$ is not identically zero on $\Gamma_1$.

Now, we plug the feedback
\begin{equation}
u(t) = 
\left<
\ \Lambda_S\mathbf{A}
\left(
  \begin{array}{c}
  \left<(y(t)\ z(t))^T,(\varphi_1\ \psi_1)^T\right>
  \\
  \left<(y(t)\ z(t))^T,(\varphi_2\ \psi_2)^T\right>
  \\
  \dots
  \\
  \left<(y(t)\ z(t))^T,(\varphi_N\ \psi_N)^T\right>
  \end{array}
\right),
\left(
  \begin{array}{c}
  \psi_1(x)\\ \psi_2(x)\\ \dots \\ \psi_N(x)
  \end{array}
\right)
\right>_N,
\end{equation} 
into equations (\ref{t15}). We obtain  the following stability result, which is commented in the forthcoming Remark~\ref{REM}. 
\begin{proposition}
\label{erop1}
The solution $(y,z)$ to the system
\begin{equation}
\label{e501}
\left\{
  \begin{array}{l}
  y_t+\nu\Delta^2 y-F_l\Delta y+\gamma\Delta z=0, \ \text{ in }(0,\infty)\times\Omega,
  \\
  \\
  z_t-\Delta z+\gamma \Delta y=0,\ \text{ in }(0,\infty)\times\Omega,
  \\
  \\
  \frac{\partial y}{\partial\mathbf{n}}=0,\ \text{on }(0,\infty)\times\Gamma,
  \\
  \\
  \frac{\partial\Delta y}{\partial\mathbf{n}}=-\frac{\gamma_0}{\nu}
  \left<\Lambda_S\ \mathbf{A}
  \left(
    \begin{array}{c}
    \left<(y\ z)^T,(\varphi_1\ \psi_1)^T
    \right>
    \\
    \\
    \left<(y\ z)^T,(\varphi_2\ \psi_2)^T\right>
    \\\dots\\
    \left<(y\ z)^T,(\varphi_N\ \psi_N)^T\right>
    \end{array}
  \right),
  \left(
    \begin{array}{c}
    \psi_1\\ \psi_2\\ \dots \\\psi_N
    \end{array}
  \right)
  \right>_N, \text{ on }(0,\infty)\times\Gamma_1
  \\
  \\
  \frac{\partial z}{\partial\mathbf{n}}=\alpha_0
  \left<\Lambda_S\ \mathbf{A}
  \left(
    \begin{array}{c}
    \left<(y\ z)^T,(\varphi_1\ \psi_1)^T\right>
    \\
    \left<(y\ z)^T,(\varphi_2\ \psi_2)^T\right>
    \\\dots\\
    \left<(y\ z)^T,(\varphi_N\ \psi_N)^T\right>
    \end{array}
  \right),
  \left(
    \begin{array}{c}
    \psi_1\\ \psi_2\\ \dots \\\psi_N
    \end{array}
  \right)
  \right>_N,\ \text{ on }(0,\infty)\times\Gamma_1,
  \\
  \\
  \frac{\partial\Delta y}{\partial\mathbf{n}}=\frac{\partial z}{\partial\mathbf{n}}=0,\ \mbox{ on }(0,\infty)\times\Gamma_2,
  \\
  \\
  y(0)=y_o,\ z(0)=z_o,
  \end{array}
\right.\ 
\end{equation}
 satisfies the exponential decay 
\begin{equation}
\label{mu}
\|(y(t)\ z(t))^T\|^2\leq C_1e^{-C_2 t}\|(y_o\ z_o)^T\|^2,\ t\geq0,
\end{equation}
for some constants  $C_1,C_2>0$.
\end{proposition}

\begin{remark}
\label{REM} 
From the practical point of view, it is important to describe 
how one can compute the first $N$ eigenvectors of the operator $\mathcal{A}$, 
since the above feedback form is expressed only in terms of those eigenvectors. 
In order to do this, we shall mainly rely on the results from the proof of Lemma \ref{erol1}.

First of all, one should compute the first, lets say $K$, 
eigenvalues and eigenfunctions of the Neumann Laplace operator, i.e.,
$$
  \Delta e_j=\mu_j e_j,\ \text{ in }\Omega;\ \frac{\partial{e_j}}{\partial\mathbf{n}}=0,\ \text{ on }\Gamma;\ j=1,2,\dots,K;
$$
for which
$$
  \mu_j\geq \frac{F_l-\gamma^2}{\nu},\ j=1,2,\dots,K.
$$ 
We have that $\mu_1=0$ and $\mu_i\neq0$ for $i=2,3,\dots,K.$

Then, for each $j=1,2,\dots,K$ one should check if the polynomial
$$
  X^2+[(F_l+1)\mu_j-\nu\mu_j^2]X-\nu\mu_j^3+(F_l-\gamma^2)\mu_j^2
$$ 
has a nonpositive root. 
It it does, we denote it by $\lambda$, and this is in fact a nonpositive eigenvalue of the operator $\mathcal{A}$. 
Let us assume that we have found $N$ such nonpositive roots, and denote them by $\lambda_i,\ i=1,2,\dots,N$. 

Hence, for each $i=1,2,\dots,N$, either $\lambda_i=0$ then, one can take 
$$
  (\varphi_i\ \psi_i)^T=\left(\sqrt{\frac{1}{2m_\Omega}}\ \sqrt{\frac{1}{2m_\Omega}}\right)^T \text{ or }(\varphi_i\ \psi_i)^T=\left(-\sqrt{\frac{1}{2m_\Omega}}\ \sqrt{\frac{1}{2m_\Omega}}\right)^T;
$$
or there exists some $j\in\left\{2,\dots,K\right\}$,  
such that the eigenvalue $\lambda_i$ can be computed as a root of the following second degree polynomial
$$
  X^2+[(F_l+1)\mu_j-\nu\mu_j^2]X-\nu\mu_j^3+(F_l-\gamma^2)\mu_j^2.
$$ 
Then, the corresponding eigenvector is given by
$$
  (\varphi_i\ \psi_i)^T=
  \left(
    \frac{\lambda_i+\mu_j}{\sqrt{(\gamma \mu_j)^2+(\lambda_i+\mu_j)^2}}e_j\  \frac{\gamma \mu_j}{\sqrt{(\gamma \mu_j)^2+(\lambda_i+\mu_j)^2}}e_j
  \right)^T.
$$
So, one can conclude that the problem reduces to finding the first $K$ eigenvalues and eigenfunctions 
of the Neumann  Laplace operator, and computing the roots of some third degree polynomials.
\end{remark}

\noindent
\ionut{\textbf{\textit{Proof of Proposition \ref{erop1}.}}} The proof follows the same lines as that one of \cite[Theorem 2.3]{ionut}, that is why, here, we only give a sketch of it. 

It will be convenient to introduce the feedbacks 
\begin{equation}
\label{pi1}
u_k(t):=
\left< \mathbf{A}
\left(
  \begin{array}{c}
  \left<(y(t),z(t)),(\phi_1,\psi_1)\right>
  \\
  \\
  \left<(y(t),z(t)),(\phi_2,\psi_2)\right>
  \\\dots\\
  \left<(y(t),z(t)),(\phi_N,\psi_N)\right>
  \end{array}
\right),
\left(
  \begin{array}{c}
  \frac{\alpha_0}{\eta_k-\lambda_1}\gianni{\psi_1}
  \\
  \frac{\alpha_0}{\eta_k-\lambda_1}\gianni{\psi_2}
  \\\dots\\
  \frac{\alpha_0}{\eta_k-\lambda_N}\gianni{\psi_N}
  \end{array}
\right)
\right>_N,\ k=1,2,\dots,N,
\end{equation}
then to define the new variables
$$
  (\tilde{y},\tilde{z}):=(y,z)-\sum_{k=1}^ND_{\eta_k}u_k,
$$ 
where $D_{\eta_k}$ were introduced in~\gianni{\eqref{defDeta}}. 
We have that $u=u_1+\dots+u_N$. 
Consequently, $(\tilde{y},\tilde{z})$ has null boundary conditions, and satisfies the following system
\begin{equation}
\label{ni1}
\left\{
  \begin{array}{l}
  \tilde{y}_t+\nu\Delta^2 \tilde{y}-F_l\Delta \tilde{y}+\gamma\Delta \tilde{z}+\mathcal{B}_1(\tilde{y},\tilde{z})=0,
  \ \text{ in }(0,\infty)\times\Omega,
  \\
  \\
  \tilde{z}_t-\Delta \tilde{z}+\gamma \Delta \tilde{y}+\mathcal{B}_2(\tilde{y},\tilde{z})=0,\ \text{ in }(0,\infty)\times\Omega,
  \\
  \\
  \pdd{\tilde{z}}=\pdd{\Delta \tilde{y}}=\pdd {\tilde{y}}=0,\ \text{ in }(0,\infty)\times\Gamma,
  \\
  \\
  (\tilde{y},\tilde{z})(0)=(\tilde{y}_o,\tilde{z}_o):=(y_o,z_o)-\sum_{k=1}^ND_{\eta_k}u_k(0),
  \end{array}
\right.\ 
\end{equation}
where 
$$\begin{aligned}
  (\mathcal{B}_1\ \mathcal{B}_2)^T=&-
  \left(
    \sum_{i=1}^N D_{\eta_i}u_i
  \right)_t
  -2\sum_{i,j=1}^N\lambda_j\left<D_{\eta_i}u_i,(\phi_j\ \psi_j)^T\right>(\phi_j\ \psi_j)^T
  \\&+\sum_{i=1}^N\eta_iD_{\eta_i}u_i
	-\delta\sum_{i=1}^N\left<D_{\eta_i}u_i,(\phi_N\ \psi_N)^T\right>(\phi_N\ \psi_N)^T.\end{aligned}
$$
(For details, see the proof of \cite[Theorem 2.1, 2.3]{ionut}.) 

Successively scalarly multiplying equation \eqref{ni1}  by $(\phi_j\ \psi_j)^T,\ j=1,2,...,N$, we obtain
(see \cite{ionut})
\begin{equation}\label{ok1}\mathcal{Z}_t=-\eta_1\mathcal{Z}+\sum_{i=2}^N(\eta_1-\eta_i)B_i\mathbf{A}\mathcal{Z}-\frac{\delta}{2}\mathcal{O},\ t>0,\end{equation}where 
$$\mathcal{Z}:=\left(\begin{array}{c}\left<(y\ z)^T,(\phi_1\ \psi_1)^T\right>\\
\left<(\tilde{y}\ \tilde{z})^T,(\phi_2\ \psi_2)^T\right>\\
...\\
...\\
\left<(\tilde{y}\ \tilde{z})^T,(\phi_N\ \psi_N)^T\right>\end{array}\right)\text{ and }\mathcal{O}=\left(\begin{array}{c}0 \\
0\\
...\\
...\\
\left<(\tilde{y}\ \tilde{z})^T,(\phi_N\ \psi_N)^T\right>\end{array}\right).$$

\rev{Clearly}, following the same arguments as in the proof of \cite[Theorem 2.3]{ionut}, due to the choice of $\delta$ and $\eta_i,\ i=1,2,...,N$ and the fact that $B_i,\ i=1,2,...,N$ are positive definite, we deduce that the $\mathbb{R}^N-$norm of $\mathcal{Z}$ is exponentially decaying, that is
$$\|\mathcal{Z}(t)\|_N\leq Ce^{-\eta_1t},\|\mathcal{Z}(0)\|_N,\ \forall t\geq0.$$This means that the first $N$ modes of the solution $(\tilde{y}\ \tilde{z})^T$ are exponentially decaying. Concerning the rest of them, we recall that they are governed by a stable linear operator, hence, they are also exponentially decaying (see the proof in \cite[Theorem 2.1,2.3]{ionut}). Then,
 recalling the definition of $\tilde{y}$ and $\tilde{z}$, 
and the fact that the first $N$ modes are exponentially decaying, 
and that the feedbacks are written only in terms of these first $N$ modes, 
we conclude immediately that (\ref{mu}) holds true.\qed

\bigskip

Recalling the notations  (\ref{t10}), 
it follows immediately by Proposition \ref{erop1} the following result concerning the linearized system of (\ref{t9})
\begin{theorem}\label{ter1}
The unique solution to the linear system 
\begin{equation}
\label{ie13}
\left\{
  \begin{array}{l}
  (\theta+l_0\varphi)_t-\Delta \theta=0,\ \text{ in }(0,\infty)\times\Omega,
  \\
  \\
  \varphi_t-\Delta \mu=0,\ \text{ in }(0,\infty)\times\Omega,
  \\
  \\
  \mu=-\nu\Delta \varphi-\varphi-\gamma_0\theta,\ \text{ in }(0,\infty)\times\Omega,
  \\
  \\
  \frac{\partial \varphi}{\partial\textbf{n}}=\frac{\partial \mu}{\partial\mathbf{n}}=0,\ \text{ on }(0,\infty)\times\Gamma,
  \\
  \\
  \frac{\partial\theta}{\partial\mathbf{n}}=u=\left< \Lambda_S \mathbf{A} \mathbf{O},\mathbf{1}\right>_N,\ \text{ on }(0,\infty)\times\Gamma_1,
  \\
  \\
  \frac{\partial\theta}{\partial\mathbf{n}}=0,\ \text{ on }(0,\infty)\times\Gamma_2,
  \\
  \\
  \theta(0)=\theta_o,\ \varphi(0)=\varphi_o,\ \text{in }\Omega,
  \end{array} 
\right.\ 
\end{equation}
satisfies the exponential decay
$$
  \|(\varphi(t)\ \theta(t))^T-(\phi_\infty\  \theta_\infty)^T\|^2
  \leq C_5e^{-C_6t}\|(\varphi_o\ \theta_o)^T-(\phi_\infty\ \theta_\infty)^T\|^2,\ 
  \forall t\geq0,
$$
for some positive constants $C_5,C_6$.  
Here,  \begin{equation}
\label{q1-bis}
\begin{aligned}
&\mathbf{O}
= \mathbf{O}(\theta,\varphi) = 
\\
&=\left(
  \begin{array}{c}           
  \left<\varphi,\varphi_1\right>
  +\left<\alpha_0(\theta+l_0\varphi),\psi_1\right>
  -\left<\varphi_\infty,\varphi_1\right>
  -\left<\alpha_0(\theta_\infty-l_0\varphi_\infty),\psi_1\right>
  \\
  \\ 
  \left<\varphi,\varphi_2\right>
  +\left<\alpha_0(\theta+l_0\varphi),\psi_2\right>
  -\left<\varphi_\infty,\varphi_2\right>
  -\left<\alpha_0(\theta_\infty-l_0\varphi_\infty),\psi_2\right>
  \\
  \\\dots \ \ \ \dots \ \ \ \dots\\ 
  \\
  \left<\varphi,\varphi_N\right>
  +\left<\alpha_0(\theta+l_0\varphi),\psi_N\right>
  -\left<\varphi_\infty,\varphi_N\right>
  -\left<\alpha_0(\theta_\infty-l_0\varphi_\infty),\psi_N\right>,
  \end{array}
\right)
\end{aligned}
\end{equation}and $\mathbf{1}$ is the vector in $\mathbb{R}^N$ with all entries equal to $1$.
\end{theorem}

\section{Conclusions}
In this paper we designed a boundary feedback stabilizing controller for stationary solutions to Chan-Hilliard system, modelling the phase separation. 
This represents the first result in literature concerning this subject. 
The feedback is with actuation only on the temperature variable, on one part of the boundary only. 
It is linear, of finite dimensional structure, given explicitly in a very simple form by (\ref{ie11}). 
It involves only the first $K\in\mathbb{N}$ eigenvalues of the Neumann Laplace operator (see Remark \ref{REM}), 
being easy for manipulation of numerical computations. 
The proofs are mainly based on the ideas from \cite{b1} and \cite{ionut}. 
The only disadvantage of this controller is that it is not robust, i.e., 
for small perturbations of the coefficients of the system, the stability is not guaranteed anymore. 
This lack of robustness may be overcomed by constructing, based on the present results, a Riccati type feedback. 
This represents the subject of a future work.


\section*{Acknowledgements}
\pier{This research activity has been performed in the framework of an
Italian-Romanian  {three-year project on ``Control and 
stabilization problems for phase field and biological systems'' financed by the Italian CNR and the Romanian Academy.} 
The present paper 
also benefits from the support of the Italian Ministry of Education, 
University and Research~(MIUR): Dipartimenti di Eccellenza Program (2018--2022) 
-- Dept.~of Mathematics ``F.~Casorati'', University of Pavia; 
the MIUR-PRIN Grant 2015PA5MP7 ``Calculus of Variations'';
the GNAMPA (Gruppo Nazionale per l'Analisi Matematica, la Probabilit\`a e le loro Applicazioni) of INdAM (Istituto Nazionale di Alta Matematica) for PC,
and the UEFISCDI project PN-III-ID-PCE-2016-0011 for IM.}


\end{document}